\documentclass[a4paper,11pt]{amsart}

\usepackage{amsfonts, latexsym, amsmath, amssymb, amsthm, amscd}
\usepackage[all]{xy}
\usepackage[english]{babel}
\usepackage[utf8]{inputenc}
\usepackage{graphicx}
\usepackage{booktabs}
\usepackage{array, tabularx}
\usepackage{textcomp}
\usepackage{tikz, tikz-cd}
\usepackage{multirow}
\usepackage{paralist}
\usepackage{comment}
\usepackage{url}
\usepackage{tensor}
\usepackage{MnSymbol}
\usepackage{color}
\usepackage{mathdots}
\usepackage[hypertexnames=false,
backref=page,
    pdftex,
    pdfpagemode=UseNone,
    breaklinks=true,
    extension=pdf,
    colorlinks=true,
    linkcolor=blue,
    citecolor=blue,
    urlcolor=blue,
]{hyperref}

\usepackage{enumitem}
\usepackage[top=1.5in, bottom=.9in, left=0.9in, right=0.9in]{geometry}

\theoremstyle{plain}
\newtheorem{theorem}{Theorem}[section]
\newtheorem*{theorem*}{Theorem}
\newtheorem{definition}[theorem]{Definition}
\newtheorem*{definition*}{Definition}
\newtheorem{lemma}[theorem]{Lemma}
\newtheorem{prop}[theorem]{Proposition}
\newtheorem*{prop*}{Proposition}

\newtheorem{rem}[theorem]{Remark}
\theoremstyle{definition}
\newtheorem{ex}[theorem]{Example}
\newtheorem*{mt*}{Main Theorem}

\def\frg{{\mathfrak{g}}}

\def\sqi{{\sqrt{-1\,}}}
\def\kod{{\operatorname{Kod}}}
\newcommand{\iso}{isomorphism}

\DeclareMathOperator{\Span}{Span}

\DeclareMathOperator{\Ker}{Ker}

\DeclareMathOperator{\im}{Im}
\DeclareMathSymbol{\Finv} {\mathord}{AMSb}{"60}
\DeclareMathOperator{\GL}{GL}
\DeclareMathOperator{\SL}{SL}
\DeclareMathOperator{\diag}{diag}

\newcommand{\Vol}{\hbox{\rm Vol}}
\newcommand\restrict[1]{\raisebox{-.5ex}{$|$}_{#1}}

\newcommand{\R}{\mathbb{R}}
\newcommand{\C}{\mathbb{C}}
\newcommand{\Z}{\mathbb{Z}}

\newcommand{\del}{\partial}
\newcommand{\delbar}{\bar{\partial}}
\newcommand{\N}{\mathbb{N}}
\numberwithin{equation}{section}

\DeclareFontFamily{U}{MnSymbolC}{}
\DeclareSymbolFont{MnSyC}{U}{MnSymbolC}{m}{n}
\DeclareFontShape{U}{MnSymbolC}{m}{n}{
    <-6>  MnSymbolC5
   <6-7>  MnSymbolC6
   <7-8>  MnSymbolC7
   <8-9>  MnSymbolC8
   <9-10> MnSymbolC9
  <10-12> MnSymbolC10
  <12->   MnSymbolC12}{}
\DeclareMathSymbol{\intprod}{\mathbin}{MnSyC}{'270}

\allowdisplaybreaks

\author{Andrea Cattaneo and Adriano Tomassini}
\address[Andrea Cattaneo, Adriano Tomassini]{
Dipartimento di Scienze Matematiche, Fisiche e Informatiche
Unit\`a di Matematica e Informatica\\
Universit\`a degli Studi di Parma\\
Parco Area delle Scienze 53/A, 43124\\
Parma, Italy}
\email{andrea.cattaneo@unipr.it}
\email{adriano.tomassini@unipr.it}

\title{\texorpdfstring{$\del\delbar$}{del-delbar}-Lemma and \texorpdfstring{$p$}{p}-K\"ahler structures on families of solvmanifolds}

\keywords{$\del\delbar$-Lemma; $p$-K\"ahler structure; Nakamura manifold; Dolbeault cohomology}
\thanks{The authors are partially supported by the project PRIN2022 “Real and Complex Manifolds: Geometry and Holomorphic Dynamics” (Project code: 2022AP8HZ9) and by GNSAGA of INdAM}
\subjclass[2020]{32J27, 57T15, 32Q99}

\date{\today}

\begin{document}
\begin{abstract}
We provide families of compact $(n+1)$-dimensional complex non K\"ahler manifolds satisfying the $\del\delbar$-Lemma, with holomoprhically trivial canonical bundle, carrying a balanced metric and with no $p$-K\"ahler structures. Such a construction extends to the completely solvable case in any dimension Nakamura's construction of low-dimensional holomorphically parallelizable solvmanifolds.
\end{abstract}

\maketitle

\tableofcontents

\section{Introduction}
The class of compact K\"ahler manifolds is a class of complex manifolds which is particularly well behaved with respect to many properties. As an example if $X$ is an $n$-dimensional compact K\"ahler manifold, from the topological point of view the Betti numbers of odd index are even and the Betti numbers of even index are positive, the de Rham complex is a differential graded algebra which is formal in the sense of Sullivan (see \cite[$\S$6, Main Theorem]{DGMS}); from the metric point of view the fundamental form of a K\"ahler metric satisfies the hard Lefschetz condition; from the complex point of view $X$ satisfies the \emph{$\del\delbar$-Lemma} (see \cite{DGMS}):
\[\Ker(\del) \cap \Ker(\delbar) \cap \im(d) = \im(\del\delbar),\]
which implies that the bidegree decomposition of $k$-forms into $(p, q)$-forms, with $p + q = k$, induces on the de Rham cohomology a Hodge structure:
\[H^k_{dR}(X, \C) \simeq \bigoplus_{p + q = k} H^{p, q}_{\delbar}(X), \qquad \overline{H^{p, q}_{\delbar}(X)} \simeq H^{q, p}_{\delbar}(X).\]
This way of stating the $\del\delbar$-Lemma (there are many others) makes it clear that its validity is a complex property of the manifold rather than a metric one. Futhermore, Harvey and Lawson in \cite[$\S$4]{HL} gave an intrinsic characterization of compact complex manifolds admitting a K\"ahler metric in terms of currents.

On the other hand, simple examples of compact complex manifolds are obtained as compact quotient of a solvable Lie group by a lattice (which is what we mean by a solvmanifold) endowed with a complex structure. Unfortunately, according to \cite[Main theorem]{Hasegawa} such manifolds are never K\"ahler unless they are finite quotients of complex tori. Nevertheless some of them carry  special metrics satisfying milder assumptions than K\"ahlerianity. In \cite{Michelsohn}, Michelsohn introduced and called \emph{balanced} a Hermitian metric with fundamental form $\omega$ such that $d(\omega^{n - 1}) = 0$, proving an intrinsic characterization of the existence of such metrics on a compact complex manifold. In \cite{AG} Abbena and Grassi  proved that any left-invariant Hermitian metric on a unimodular complex Lie group is balanced. In view of the Theorem of Wang \cite[Theorem 1]{W}, the result by Abbena and Grassi implies that any compact holomorphically parallelizable complex manifold admits a balanced metric. Sullivan introduced in \cite[Definition I.3]{S} the notion of {\em transversality} in the context of cone structures, namely a continuos field of cones of $p$-vectors on a manifold. Following Sullivan, Alessandrini and Andreatta gave the definition of \emph{$p$-K\"ahler manifold} in \cite[Definition 1.11]{AA}. Both K\"ahler and balanced metrics have a nice feature with respect to the complex submanifolds of the ambient space: the K\"ahler form of a K\"ahler metric restricts to a volume form on any $1$-dimensional submanifold, while for a balanced metric the form $\omega^{n - 1}$ restricts to a volume form on any $(n - 1)$-dimensional submanifold. By ``interpolation'' of these two extremes we have the concept of a $p$-K\"ahler manifold.

Roughly speaking, a $p$-K\"ahler manifold is a complex manifold $X$ endowed with $p$-K\"ahler form, namely a $d$-closed real $(p, p)$-form $\omega$ which restricts to a volume form on every $p$-dimensional submanifold of $X$. The cases $p = 1$ and $p = n - 1$ are peculiar because in these cases the $p$-K\"ahler form is respectively the K\"ahler form of a K\"ahler metric or the $(n - 1)^{\text{th}}$ power of the K\"ahler form of a balanced metric, hence they correspond to K\"ahler manifolds and balanced manifolds respectively. In contrast, the intermediate K\"ahler degrees $p = 2, \ldots, n - 2$ are genuinely complex and not metric invariants of the manifold, since a $p$-K\"ahler manifold whose $p$-K\"ahler form is the $p^\text{th}$ power of the K\"ahler form of a metric is necessarily K\"ahler (see \cite[Proposition 2.1]{AB1}). This shows the relevance of $p$-K\"ahler manifolds in non-K\"ahler geometry, as they stress more the complex geometrical aspects rather than the metric ones.


As far as we know, all known examples of manifolds satisfying the $\del\delbar$-Lemma carry a balanced metric. In this paper we want to investigate the possible relations between $p$-K\"ahler structures and $\del\delbar$-Lemma, and provide families of $p$-K\"ahler manifolds. More precisely, we prove the following:

\begin{theorem}\label{thm: main}
Let $M \in \operatorname{SL}(n, \Z)$ be digonalizabile over $\R$ with positive eigenvalues and let $P \in \operatorname{GL}(n, \R)$ be such that $PMP^{-1}$ is diagonal. Let $\tau \in \R \smallsetminus \{ 0 \}$. Then there exists an $(n + 1)$-dimensional compact complex manifold $N = N_{M, P, \tau}$ such that:
\begin{enumerate}
\item\label{item: construction} $N$ is a compact solvmanifold endowed with an invariant complex structure;
\item\label{item: balanced and p-kahler} $N$ admits a balanced metric and is not $p$-K\"ahler for any $p = 1, \ldots, n - 1$;
\item\label{item: kodaira} $N$ has holomorphically trivial canonical bundle, in particular the Kodaira dimension of $N$ is $0$;
\item\label{item: del-delbar-lemma} if moreover $\tau$ satisfies the assumption \eqref{eq: technical condition}, then $N$ satisfies the $\del\delbar$-Lemma and has no K\"ahler structure.
\end{enumerate}
\end{theorem}

This theorem provides examples of $(n + 1)$-dimensional manifolds of Kodaira dimension $0$ satisfying the $\del\delbar$-Lemma but with no $p$-K\"ahler structures for $p \neq n$. More in general, since $N$ is a compact solvmanifold of completely solvable type different from a torus, by Hasegawa \cite[Main Theorem]{Hasegawa} has no K\"ahler structure. For other results on $p$-K\"ahler manifolds see \cite{RWZ,HMT,FM} and the references therein.

The construction of $p$-K\"ahler manifolds and of manifolds satisfying the $\del\delbar$-Lemma is a very hard (but challenging) problem. The first examples of $p$-K\"ahler manifolds are the $\eta\beta_{2n + 1}$ by Alessandrini and Bassanelli (see \cite[$\S$4]{AB1}): they are $(2n + 1)$-dimensional nilmanifolds which are $p$-K\"ahler for every $p \geq n + 1$. But since they are nilmanifolds (not even homeomorphic to tori), they can not satisfy the $\del\delbar$-lemma. Observe that it is possible to have manifolds satisfying the $\del\delbar$-Lemma which are also $p$-K\"ahler for some $p$, e.g., in \cite[$\S$4]{AB2} there is an example of a non-K\"ahlerian Moishezon manifold (which then satisfies the $\del\delbar$-Lemma) of dimension $5$ which is $p$-K\"ahler for every $p \geq 2$.

The plan of the paper is as follows. In Section \ref{preliminaries} we recall the basic definitions concerning $p$-K\"ahlerianity and the $\del\delbar$-Lemma. In Section \ref{sect: construction} we show the construction of Nakamura manifolds, which extends to the completely solvable case in any dimension  Nakamura's construction of low-dimensional holomorphically parallelizable solvmanifolds in \cite[p.90 Case III-(3a)]{N}. This section also proves the first item of Theorem \ref{thm: main}. Observe that our construction ensures the existence of a lattice in a completely solvable Lie group, which is a non-trivial fact. Item \eqref{item: balanced and p-kahler} of Theorem \ref{thm: main} is shown in Section \ref{sect: balanced} and \ref{sect: p-kahler}; item \eqref{item: kodaira} in Section \ref{sect: kodaira}. To show that item \eqref{item: del-delbar-lemma} holds in Section \ref{sect: del-delbar-lemma} we need to compute the Dolbeault and de Rham cohomology of our manifolds. This is done in Section \ref{sect: dolbeault} and \ref{sect: de Rham} respectively, using a powerful tool by Kasuya (see \cite{Kasuya}) and a classical result by Hattori (see \cite{Hattori}). In \cite{Kasuya1} a similar situation is studied, more in detail Kasuya gives a general framework in the context of solvmanifolds obtained as compact quotients of $G = \C^n \ltimes_\phi \C^m$ and proved that under suitable assumptions it is possible to define a Hermitian metric for which the Hodge symmetry and decomposition hold for Dolbeault and de Rham forms; this implies the validity of the $\del\delbar$-Lemma. However, our construction gives different examples from those obtained by Kasuya. Finally, Section \ref{sect: examples} is devoted to the construction of explicit examples of Nakamura manifolds.

\vspace{0.2cm}\noindent

{\em \underline{Acknowledgments:}} The authors would like to thank Hisashi Kasuya for bringing to their attention his paper \cite{Kasuya1}.

\section{Preliminaries}\label{preliminaries}
In this Section we will fix the notation we will use in the paper and we recall the notions of $p$-{\em K\"ahler structure} and of $\del\delbar$-{\em Lemma}. We start by some linear algebra definitions.

Let $V$ be a real $2n$-dimensional vector space endowed with a complex structure $J$, that is an automorphism $J$ of $V$ satisfying $J^2=-\hbox{\rm id}_V$. Let $V^*$ be the dual space of $V$ and denote by the same symbol the complex structure on $V^*$ naturally induced by $J$ on $V$. Then the complexified 
$V^*_\C$ decomposes as the direct sum of the $\pm \sqi$-eigenspaces, $V^{1,0}$ and $V^{0,1}$, of the extension of $J$ to $V^*_\C$, given by
\[\begin{array}{l}
V^{1,0} = \{ \varphi \in V^*_\C \,\,\,\vert\,\,\, J \varphi = \sqi \varphi \} =  \{ \alpha - \sqi J \alpha \,\,\,\vert\,\,\, \alpha \in V^* \}\\[10pt]
V^{0,1} = \{ \psi \in V^*_\C \,\,\,\vert\,\,\, J \psi = -\sqi \psi \} = \{ \beta + \sqi J \beta \,\,\,\vert\,\,\, \beta \in V^* \},
\end{array}\]
that is 
\[V^*_\C = V^{1,0} \oplus V^{0,1}.\]
According to the above decomposition, the space $\Lambda^r(V^*_\C)$ of complex $r$-covectors on $V_\C$ decomposes as
\[\Lambda^r(V^*_\C) = \bigoplus_{p + q = r} \Lambda^{p,q}(V^*_\C),\]
where
\[\Lambda^{p,q}(V^*_\C) = \Lambda^p(V^{1,0}) \otimes \Lambda^q(V^{0,1}).\]

Let $\{ \varphi^1, \ldots, \varphi^n \}$ be a basis of $V^{1,0}$, then
\[\{ \varphi^{i_1} \wedge \cdots \wedge \varphi^{i_p} \wedge \overline{\varphi^{j_1}} \wedge \cdots \wedge \overline{\varphi^{j_q}} \,\,\,\vert\,\,\, 1 \leq i_1 < \cdots < i_p \leq n ,\,\, 1 \leq j_1 < \cdots < j_q \leq n \}\]
is a basis of $\Lambda^{p,q}(V^*_\C)$. Set $\sigma_p = \sqi^{p^2} 2^{-p}$. Then, given any $\varphi \in \Lambda^{p, 0}(V^*_\C)$ we have that
\[\overline{\sigma_p \, \varphi \wedge \overline{\varphi}} = \sigma_p \, \varphi \wedge \overline{\varphi},\]
that is $\sigma_p \, \varphi \wedge \overline{\varphi}$ is a real $(p, p)$-form. We will denote by
\[\Lambda_{\R}^{p, p}(V^*_\C) = \{ \psi \in \Lambda^{p, p}(V^*_\C) \,\,\,\vert\,\,\, \psi = \overline{\psi}\},\]
the space of real $(p,p)$-covectors.
Note that the complex structure $J$ acts on the space of real $k$-covectors $\Lambda^k(V^*)$ by setting, for any given $\alpha\in \Lambda^k(V^*)$,
\[J\alpha (V_1, \ldots, V_k) = \alpha(JV_1, \ldots, JV_k).\]
It is immediate to check that if $\psi \in \Lambda_{\R}^{p, p}(V^*_\C)$ then $J\psi = \psi$. For $k = 2$, the converse also holds.

Denoting by 
\[\Vol = \left( \frac{\sqi}{2} \varphi^1 \wedge \overline{\varphi^1} \right) \wedge \cdots \wedge \left( \frac{\sqi}{2} \varphi^n \wedge \overline{\varphi^n} \right),\]
we obtain that 
\[\Vol = \sigma_n \varphi^1 \wedge \cdots \wedge \varphi^n \wedge \overline{\varphi^1} \wedge \cdots \wedge \wedge \overline{\varphi^n},\]
that is $\Vol$ is a volume form on $V$. A real $(n, n)$-form $\psi$ is said to be {\em positive}, respectively {\em strictly positive}, if 
\[\psi = c \Vol,\] 
where $c \geq 0$, respectively $c > 0$.

By definition, $\psi \in \Lambda^{p,0}(V^*_\C)$ is said to be {\em simple} or {\em decomposable} if
\[\psi = \eta^1 \wedge \cdots \wedge \eta^p,\]
for suitable $\eta^1, \ldots, \eta^p \in V^{1,0}$. Let $\Omega \in \Lambda_{\R}^{p, p}(V^*_\C)$. Then $\Omega$ is said to be {\em transverse} if, given any non-zero simple $(n - p)$-covector $\psi$, the real $(n, n)$-form
\[\Omega \wedge \sigma_{n - p} \psi \wedge \overline{\psi}\]
is strictly positive. 

Let $(M, J)$ be a complex manifold of real dimension $2n$, that is $M$ is a $2n$-dimensional smooth manifold endowed with a smooth $(1, 1)$-tensor field $J$ satisfying $J^2 = -\hbox{id}_{TM}$ and the integrability condition $N_J = 0$, where $N_J$ is the Nijenhuis tensor of $J$ defined, for any given pair of vector fields $X, Y$ on $M$, as
\[N_J(X, Y) = [JX, JY] - [X, Y] - J[JX, Y] - J[X, JY].\]
The complex structure $J$ extends to the complexified $T_\C M$ of the tangent bundle $TM$ and $T_\C M = T^{(1, 0)}M \oplus T^{(0, 1)}M$, where $T^{(1, 0)}M$, respectively $T^{(0, 1)}M$ denote the $\sqi$-eigenbundle, respectively the $-\sqi$-eigenbundle, of $J$. Accordingly, the bundle of complex $k$-forms on $M$ decomposes as the direct sum of the bundle $\Lambda^{p,q}M$, that is
\[\Lambda^k(M,\C) = \bigoplus_{p + q = k} \Lambda^{p, q}M.\]
We will denote by $A^{p, q}(M)$ the space of complex $(p, q)$-forms, which is the space of smooth sections of the bundle $\Lambda^{p,q}(M)$ and by $A_{\R}^{p, p}(M)$ the space of real $(p, p)$-forms. With respect to the above decomposition, the exterior differential $d$ acting on $A^{p, q}(M)$ splits as $d = \del + \delbar$, where
\begin{equation}
\del\restrict{A^{p, q}(M)} := \pi^{p + 1, q} \circ d \restrict{A^{p, q}(M)}, \qquad \delbar\restrict{A^{p, q}(M)} := \pi^{p , q + 1} \circ d\restrict{A^{p,q}(M)},
\end{equation}
are the projections of $d(A^{p, q}(M))$ onto $A^{p + 1, q}(M)$ and $A^{p, q + 1}(M)$ respectively. Then the {\em $(p, q)$-Dolbeault cohomology group} of the complex manifold $M$ is defined as
\[H^{p, q}_{\delbar}(M) = \frac{\Ker \delbar \cap A^{p, q}(M)}{\im \delbar \cap A^{p, q}(M)}.\]
The notion of positivity on complex vector spaces can be transferred pointwise to complex manifolds. 

\begin{definition}
Let $(M, J)$ be a complex manifold of real dimension $2n$ and let $1 \leq p \leq n$. A $p$-{\em K\"ahler structure} on $(M, J)$ is a closed real transverse $(p, p)$-form $\Omega$, that is $\Omega$ is $d$-closed and, at every $x \in M$, $\Omega_x \in \Lambda^{p, p}_{\R}(T_x^*M)$ is transverse. The triple $(M, J, \Omega)$ is said to be a {\em p-K\"ahler manifold}.
\end{definition}

Finally, we recall that, according to Deligne, Griffiths, Morgan and Sullivan, a compact complex manifold $(M, J)$ is said to satisfy the \emph{$\del\delbar$-lemma} (see \cite[5.21, Proposition 5.17 and Lemma 5.15 (a)]{DGMS}) if
\[\Ker\del \cap \Ker\delbar \cap \im (d) = \im\del\delbar.\] 
The above condition is equivalent to say, (see \cite[5.21 and Proposition 5.17]{DGMS} and also \cite[Corollary 8 and Definition 4]{ASTT}) that for every $k\in\N$, the following conditions hold:
\begin{enumerate}
\item[(1)] There exist \iso s
\[H^{p, q}_{\delbar}(M) \simeq \overline{H^{q, p}_{\delbar}(M)},\,  p + q = k, \quad \text{and} \quad H^k_{dR}(M) \simeq \bigoplus_{p + q = k}H^{p, q}_{\delbar}(M).\]
\item[(2)] Every class in $H^{p, q}_{\delbar}(M)$, $p + q = k$, admits a representative $\omega$ with $\partial \omega = 0$ and $\delbar \omega = 0$, i.e., $d\omega = 0$. Moreover, the assignment $\omega \longmapsto \bar{\omega}$ induces the first \iso\ above.
\item[(3)] Every class in $H^k_{dR}(M)$ admits a representative $\omega$ which may be written as $\omega = \sum_{p + q = k} \omega^{p, q}$, where $\omega^{p, q}$ is a  $(p, q)$-form with $d\omega^{p, q} = 0$. Moreover, the assignment $\omega \longmapsto (\omega^{p, q})_{p + q = k}$ induces the second \iso\ above.
\end{enumerate}


\section{Construction of Nakamura manifolds}\label{sect: construction}
One of the first and most known examples of $3$-dimensional compact complex non K\"ahler manifold has been given by Nakamura in \cite[p.90 Case III-(3a)]{N}. His examples are obtained as a compact quotient of a complex solvable (non nilpotent) Lie group $G$ and thus they are holomorphically parallelizable. In contrast to the nilpotent case, one of the main difficulties is to construct a lattice in $G$. Starting from his construction and the one given by Kasuya in \cite[Example 5.1]{Kasuya}, we extend such a construction to complex dimension $(n+1)$. As we will see, our examples are not holomorphically parallelizable, but of completeley solvable type. 

Let $M \in \SL(n, \Z)$ be a matrix with positive eigenvalues and diagonalizable over $\R$. Then there exist a matrix $P \in \GL(n, \R)$ and numbers $\lambda_1, \ldots, \lambda_n \in \R$ such that
\[P M P^{-1} = \diag\left( e^{\lambda_1}, \ldots, e^{\lambda_n} \right) = \Lambda.\]

\begin{rem}\label{rem: sum lambda_i vanishes}
As $M \in \SL(n, \Z)$ we deduce that
\[\sum_{i = 1}^n \lambda_i = 0.\]
\end{rem}

We define the group homomorphism
\begin{equation}
\begin{array}{rccl}
\rho: & \C & \longmapsto & \GL(n, \C)\\
 & w & \longmapsto & \diag\left( e^{\frac{1}{2}\lambda_1 (w + \bar{w})}, \ldots, e^{\frac{1}{2}\lambda_n (w + \bar{w})} \right)
\end{array}
\end{equation}
and consider the semidirect product $(G_M; *) = \C \ltimes_\rho \C^n$. Explicitly, the group law $*$ is given by
\begin{equation}\label{eq: multiplication}
(w', z_1', \ldots, z_n') * ( w, z_1, \ldots, z_n) = \left(w' + w, z_1' + e^{\frac{1}{2}\lambda_1 (w + \bar{w})} z_1, \ldots, z_n' + e^{\frac{1}{2}\lambda_n (w + \bar{w})} z_n \right).
\end{equation}

Fix $\tau \in \R \smallsetminus \{ 0 \}$ and let
\[\Gamma'_\tau = \Z \oplus \sqi \tau \cdot \Z \subseteq \C, \qquad \Gamma''_P = P \cdot \Z^n \oplus \sqi P \cdot \Z^n \subseteq \C^n.\]
Observe that $\rho(1) = \Lambda$, which is invertible, so $\rho(1) \cdot \Gamma''_P = \Gamma''_P$. It follows that $\rho(m) \cdot \Gamma''_P = \Gamma''_P$ for every $m \in \Z$. Similarly, $\rho(\sqi \tau) = I_n$ and so $\rho(\sqi \tau m) \cdot \Gamma''_P = \Gamma''_P$ for every $m \in \Z$. It follows that the action of $\Gamma'_\tau$ on $\C^n$ via $\rho$ preserves the subgroup $\Gamma''_P$ and so
\[\Gamma_{P, \tau} = \Gamma'_\tau \ltimes_\rho \Gamma''_P\]
is a subgroup of $G_M$.

We define the \emph{Nakamura manifold} associated to $(M, P, \tau)$
\[N_{M, P, \tau} = \Gamma_{P, \tau} \backslash G_M\]
which is a compact complex manifold of dimension $n + 1$. In the sequel we will denote it simply by $N$, keeping the dependence from $M$, $P$ and $\tau$ as understood.

\begin{rem}
Up to now we made no assumption on the eigenvalues $\lambda_i$ of $M$. Assume that $t$ of them vanish: it follows from \eqref{eq: multiplication} that we can extract from $G_M$ a direct factor of $\C^t$, i.e., that $G_M = G'_{M'} \times \C^t$, where $G'_{M'}$ is constructed as a semidirect product $G'_{M'} = \C \ltimes_{\rho'} \C^{n - t}$ as in this section. As a consequence, $N_{M, P, \tau}$ also splits as
\[N_{M, P, \tau} \simeq N'_{M', P', \tau} \times T,\]
where $T$ is a $t$-dimensional complex torus. For this reason, from now \emph{we will assume that $\lambda_i \neq 0$ for every $i = 1, \ldots, n$}.
\end{rem}

Observe that $G_M$ is both a complex manifold and a group, but it is not a complex Lie group as by \eqref{eq: multiplication} the multiplication $*$ is not holomorphic.

The $(1, 0)$-forms on $G_M$
\begin{equation}
\varphi^0 = dw\, \qquad \varphi^i = e^{-\frac{1}{2}\lambda_i (w + \bar{w})} dz_i \qquad i = 1, \ldots, n,
\end{equation}
define a set of $n + 1$ forms on $G_M$ which are $\Gamma_{P, \tau}$-invariant, hence descend to a set of $(1, 0)$-forms on $N_{M, P, \tau}$.

This result is immediate.
\begin{lemma}
The following structure equations hold on $N_{M, P, \tau}$:
\begin{equation}
d\varphi^0 = 0, \qquad d\varphi^i = -\frac{1}{2} \lambda_i (\varphi^0 + \bar{\varphi}^0) \wedge \varphi^i \qquad i = 1, \ldots, n.
\end{equation}
\end{lemma}

As a consequence $\{ \varphi^0, \varphi^1, \ldots, \varphi^n \}$ is a smooth trivialization of $TN$, which is not holomorphic as we are assuming that $\lambda_i \neq 0$ for every $i = 1, \ldots, n$. In particular, $N$ is not holomorphically parallelizable.

\section{Metric and cohomological properties of Nakamura manifolds}\label{sect: properties of Nakamura manifolds}

Throughout this section, we fix a Nakamura manifold $N$ associated to $(M, P, \tau)$ as constructed in Section \ref{sect: construction} and we show that Theorem \ref{thm: main} holds.

\subsection{Balanced metrics on \texorpdfstring{$N$}{N}}\label{sect: balanced}
We want to show that $N$ carries an invariant balanced metric, proving part of item \eqref{item: balanced and p-kahler} in Theorem \ref{thm: main}. Let
\begin{equation}\label{eq: metric 1}
\omega = \sqi \sum_{i = 0}^n \varphi^i \wedge \bar{\varphi}^i.
\end{equation}
Then
\[\omega^n = (\sqi)^n n! \sum_{i = 0}^n \varphi^0 \wedge \bar{\varphi}^0 \wedge \ldots \widehat{\varphi^i \wedge \bar{\varphi}^i} \wedge \ldots \varphi^n \wedge \bar{\varphi}^n,\]
and so
\begin{equation}\label{eq: metric 1 is balanced}
d\omega^n = (\sqi)^n n! \sum_{i = 0}^n \sum_{j \neq i} \varphi^0 \wedge \bar{\varphi}^0 \wedge \ldots d(\varphi^j \wedge \bar{\varphi}^j) \wedge \ldots \widehat{\varphi^i \wedge \bar{\varphi}^i} \wedge \ldots \varphi^n \wedge \bar{\varphi}^n = 0
\end{equation}
as $d\varphi^0 = d\bar{\varphi}^0 = 0$ and
\[d(\varphi^j \wedge \bar{\varphi}^j) = -\lambda_j \varphi^0 \wedge \varphi^j \wedge \bar{\varphi}^j + \lambda_j \varphi^j \wedge \bar{\varphi}^0 \wedge \bar{\varphi}^j.\]

\begin{prop}
The manifold $N$ carries an invariant balanced metric.
\end{prop}
\begin{proof}
Let $\omega$ be as in \eqref{eq: metric 1}. If we show that $\omega$ is positive, then $\omega$ is the K\"ahler form of an invariant balanced metric on $N$ by \eqref{eq: metric 1 is balanced}. Let $V_0, V_1, \ldots, V_n$ be the frame of $T^{1, 0}N$ dual to $\varphi^0, \varphi^1, \ldots, \varphi^n$ and let
\[V = \sum_{j = 0}^n f_j V_j \in \Gamma(N, T^{1, 0}N)\]
be a smooth vector field on $N$. Then
\[-\sqi \omega(V, \bar{V}) = \sum_{i = 0}^n |f_i|^2 \geq 0\]
and equality holds if and only if $V = 0$. This shows that
\[h(\cdot, \cdot) = \omega(J\cdot, \cdot) - \sqi \omega(\cdot, \cdot)\]
is a balanced metric on $N$.
\end{proof}

\subsection{\texorpdfstring{$p$}{p}-K\"ahlerianity of \texorpdfstring{$N$}{N}}\label{sect: p-kahler}
In this section we show that $N$ is not $p$-K\"ahler for some $1 < p < n$, thus showing the second part of item \eqref{item: balanced and p-kahler} of Theorem \ref{thm: main}.

\begin{rem}
Observe that $N$ can not be $1$-K\"ahler as being $1$-K\"ahler amounts to be K\"ahler, while $N$ is $n$-K\"ahler as being $n$-K\"ahler amounts to be a balanced manifold (racall that $n = \dim(X) - 1$).
\end{rem}

Let $1 \leq i_1 < \ldots < i_{n - p} \leq n$ be a multiindex of length $n - p$, with $0 \leq p \leq n - 1$. We compute that
\begin{equation}\label{eq: phi is exact}
\begin{array}{c}
d(\varphi^0 \wedge \varphi^{i_1} \wedge \bar{\varphi}^{i_1} \wedge \ldots \wedge \varphi^{i_{n - p}} \wedge \bar{\varphi}^{i_{n - p}}) =\\
= -\varphi^0 \wedge \sum_{j = 1}^{n - p} \varphi^{i_1} \wedge \bar{\varphi}^{i_1} \wedge \ldots \wedge d(\varphi^{i_j} \wedge \bar{\varphi}^{i_j}) \wedge \ldots \wedge \varphi^{i_{n - p}} \wedge \bar{\varphi}^{i_{n - p}} =\\
= -\varphi^0 \wedge \sum_{j = 1}^{n - p} \varphi^{i_1} \wedge \bar{\varphi}^{i_1} \wedge \ldots \wedge (-\lambda_{i_j} \varphi^0 \wedge \varphi^{i_j} \wedge \bar{\varphi}^{i_j} - \lambda_{i_j} \bar{\varphi}^0 \wedge \varphi^{i_j} \wedge \bar{\varphi}^{i_j}) \wedge \ldots \wedge \varphi^{i_{n - p}} \wedge \bar{\varphi}^{i_{n - p}} =\\
= \left( \sum_{j = 1}^{n - p} \lambda_{i_j} \right) \varphi^0 \wedge\bar{\varphi}^0 \wedge \varphi^{i_1} \wedge \bar{\varphi}^{i_1} \wedge \ldots \wedge \varphi^{i_{n - p}} \wedge \bar{\varphi}^{i_{n - p}}.
\end{array}
\end{equation}

The following lemma will be useful to show that $N$ is not $p$-K\"ahler.

\begin{lemma}\label{lemma: sum 0 on multiindices}
Let $n \in \Z_{\geq 2}$, $\lambda_1, \ldots, \lambda_n \in \R$ and $1 \leq k \leq n - 1$. Denote by $I = (1 \leq i_1 < \ldots < i_k \leq n)$ a multiindex of length $k$. If $\sum_{i_j \in I} \lambda_{i_j} = 0$ for every multiindex $I$ of length $k$ then $\lambda_1 = \ldots = \lambda_n = 0$.
\end{lemma}
\begin{proof}
This is of course the case if $k = 1$. Assume now that $k \geq 2$ (which implies that $n \geq 3$). Choose two different indices $i, j \in \{ 1, \ldots, n \}$; as $k \neq n$ we can construct two different multiindices $I$ and $J$ of length $k$ as follows: $I$ contains $i$ but not $j$, and $J$ coincides with $I$ except that $i$ is replaced by $j$. As a consequence
\[\sum_{i_j \in I} \lambda_{i_j} = 0 = \sum_{i_j \in J} \lambda_{i_j} \Longrightarrow \lambda_i = \lambda_j.\]
This shows that $\lambda_1 = \ldots = \lambda_n$. To show that all of them vanish it suffices to show that one vanishes: choose any multiindex $I$ of length $k$, so
\[\sum_{i_j \in I} \lambda_{i_j} = 0 \Longrightarrow k \cdot \lambda_{i_1} = 0 \Longrightarrow \lambda_{i_1} = 0.\]
\end{proof}

\begin{prop}\label{prop: N non p-Kahler}
Let $1 \leq p \leq n - 1$. Then $N$ is not $p$-K\"ahler.
\end{prop}
\begin{proof}
Assume on the contrary that $N$ is $p$-K\"ahler, and let $\omega$ be a $p$-K\"ahler form. Since $N$ is not a torus, by Lemma \ref{lemma: sum 0 on multiindices} there is a multiindex $I = (1 \leq i_1 < \ldots < i_{n - p} \leq n)$ of length $k = n - p$ such that $\sum_{i_j \in I} \lambda_{i_j} \neq 0$. Let
\[\begin{array}{rl}
\varphi = & (\sqi)^{p^2} 2^{-p} \cdot \varphi^0 \wedge \varphi^{i_1} \wedge \ldots \wedge \varphi^{i_{n - p}} \wedge \bar{\varphi}^0 \wedge \bar{\varphi}^{i_1} \wedge \ldots \wedge \bar{\varphi}^{i_{n - p}} =\\
= & (-1)^{\frac{(n - p + 1)(n - p)}{2}} (\sqi)^{p^2} 2^{-p} \cdot \varphi^0 \wedge\bar{\varphi}^0 \wedge \varphi^{i_1} \wedge \bar{\varphi}^{i_1} \wedge \ldots \wedge \varphi^{i_{n - p}} \wedge \bar{\varphi}^{i_{n - p}} =\\
= & (-1)^{\frac{(n - p + 1)(n - p)}{2}} (\sqi)^{p^2} 2^{-p} \left( \sum_{j = 1}^{n - p} \lambda_{i_j} \right)^{-1} \cdot d(\varphi^0 \wedge \varphi^{i_1} \wedge \bar{\varphi}^{i_1} \wedge \ldots \wedge \varphi^{i_{n - p}} \wedge \bar{\varphi}^{i_{n - p}}),
\end{array}\]
which by \eqref{eq: phi is exact} is a non-trivial exact $(n - p + 1, n - p + 1)$-form. It follows on the one hand that
\[\int_N \omega \wedge \varphi > 0\]
as $\omega$ is a $p$-K\"ahler form, while on the other hand that
\[\int_N \omega \wedge \varphi = \int_N d(\omega \wedge \varphi^0 \wedge \varphi^{i_1} \wedge \bar{\varphi}^{i_1} \wedge \ldots \wedge \varphi^{i_{n - p}} \wedge \bar{\varphi}^{i_{n - p}}) = 0\]
as $\omega$ is closed.
\end{proof}

\subsection{Kodaira dimension}\label{sect: kodaira}
As we said, the manifold $N$ are not holomorphically parallelizable, nonetheless their canonical bundle is holomorphically trivial. This proves item \eqref{item: kodaira} in Theorem \ref{thm: main}.

\begin{prop}
We have $\omega_N \simeq \mathcal{O}_N$, in particular $\kod(N) = 0$.
\end{prop}
\begin{proof}
The smooth $(n + 1, 0)$-form $\varphi^0 \wedge \varphi^1 \wedge \ldots \wedge \varphi^n$ is a smooth trivialization of the complex vector bundle $\omega_N$ and it satisfies
\[\begin{array}{rl}
\delbar\left( \varphi^0 \wedge \varphi^1 \wedge \ldots \wedge \varphi^n \right) = & -\varphi^0 \wedge \delbar \left( \varphi^1 \wedge \ldots \wedge \varphi^n \right) =\\
= & -\varphi^0 \wedge \sum_{i = 1}^n (-1)^{i - 1} \varphi^1 \wedge \ldots \wedge \underbrace{\left( -\frac{1}{2} \lambda_i \bar{\varphi}^0 \wedge \varphi^i \right)}_{i\text{-th place}} \wedge \ldots \wedge \varphi^n =\\
= & \frac{1}{2} \left( \sum_{i = 1}^n \lambda_i \right) \varphi^0 \wedge \bar{\varphi}^0 \wedge \varphi^1 \wedge \ldots \wedge \varphi^n =\\
= & 0
\end{array}\]
by Remark \ref{rem: sum lambda_i vanishes}. This shows that $\omega_N$ is holomorphically trivial.
\end{proof}

In dimension $6$, the Lie algebra of $G$ appeared as $\mathfrak{g}_1$ in the list by Fino, Otal and Ugarte in \cite[Theorem 2.8]{FOU}, where a  classification of holomorphically trivial $6$-dimensional compact complex solvmanifolds is provided.
\subsection{Dolbeault cohomology}\label{sect: dolbeault}
In this section we compute the Dolbeault cohomology of the generalized Nakamura manifolds $N$ introduced above.

It is easy to see that our construction matches the requirements of \cite[Assumption 1.1]{Kasuya}. As a consequence, by \cite[Corollary 4.2]{Kasuya} the Dolbeault cohomology of $N$ can be computed by a suitable sub-differential bigraded algebra $B^{p, q}$ of the differential bigraded algebra
\[\Lambda^{p, q} N = \bigwedge^p (T^* N)^{1, 0} \otimes \bigwedge^q (T^* N)^{0, 1}.\]

In order to describe $B^{p, q}$, let us introduce some notations. Let $I = (i_1, \ldots, i_p)$ and $J = (j_1, \ldots, j_q)$ be multi-indices of length $p$ and $q$ respectively, namely
\[1 \leq i_1 < \ldots < i_p \leq n, \qquad 1 \leq j_1 < \ldots < j_q \leq n,\]
we will denote
\[\varphi^I = \varphi^{i_1} \wedge \ldots \wedge \varphi^{i_p}, \qquad \bar{\varphi}^J = \bar{\varphi}^{j_1} \wedge \ldots \wedge \bar{\varphi}^{j_q}\]
and
\[c_{IJ} = \sum_{\sigma = 1}^p \lambda_{i_\sigma} + \sum_{\tau = 1}^q \lambda_{j_\tau}.\]
With these notations, $B^{p, q}$ is generated by the subset of $(p, q)$-forms among those of the form
\begin{equation}\label{eq: generators of B^{p, q}}
\begin{array}{llll}
f_{IJ} \cdot \varphi^I \wedge \bar{\varphi}^J & \text{with $I$, $J$ multiindices such that} & |I| = p, & |J| = q;\\
f_{IJ} \cdot \varphi^0 \wedge \varphi^I \wedge \bar{\varphi}^J & \text{with $I$, $J$ multiindices such that} & |I| = p - 1, & |J| = q;\\
f_{IJ} \cdot \bar{\varphi}^0 \wedge \varphi^I \wedge \bar{\varphi}^J & \text{with $I$, $J$ multiindices such that} & |I| = p, & |J| = q - 1;\\
f_{IJ} \cdot \varphi^0 \wedge \bar{\varphi}^0 \wedge \varphi^I \wedge \bar{\varphi}^J & \text{with $I$, $J$ multiindices such that} & |I| = p - 1, & |J| = q - 1;
\end{array}
\end{equation}
corresponding to multiindices $I$, $J$ such that the function
\[f_{IJ}(w) = e^{-\frac{1}{2} c_{IJ} (w - \bar{w})}\]
satisfies $f_{IJ}(\gamma) = 1$ for every $\gamma \in \Gamma'$.

It is an immediate computation that
\begin{equation}\label{eq: delbar f_IJ}
\delbar f_{IJ} = \frac{1}{2} c_{IJ} f_{IJ} \bar{\varphi}^0.
\end{equation}

Recall that $\Gamma'_\tau = \Z \oplus \Z \cdot \sqi \tau$ for $\tau \in \R \smallsetminus \{ 0 \}$. We can characterize the generators of $B^{p, q}$ in terms of the parameter $\tau$ defining the lattice $\Gamma'$ as follows.
\begin{lemma}
$B^{p, q}$ is generated by the $(p, q)$-forms in \eqref{eq: generators of B^{p, q}} corresponding to multiindices $I$, $J$ such that
\[\tau \cdot c_{IJ} \in 2\pi \cdot \Z\]
\end{lemma}
\begin{proof}
Let $I$, $J$ be a pair of multiindices, then $f_{IJ}(\gamma) = 1$ for every $\gamma \in \Gamma'_\tau$ if and only if $f_{IJ}(\sqi \tau) = 1$. By definition this means that $e^{-\sqi \tau c_{IJ}} = 1$ and the lemma follows.
\end{proof}

\begin{prop}\label{prop: delbar trivial}
We have $\delbar (f_{IJ} \cdot \varphi^I \wedge \bar{\varphi}^J) = 0$. In particular, this implies that $\delbar \equiv 0$ on $B^{p, q}.$
\end{prop}
\begin{proof}
Using \eqref{eq: delbar f_IJ} we see that
\[\begin{array}{rl}
\delbar (f_{IJ} \cdot \varphi^I \wedge \bar{\varphi}^J) = & \frac{1}{2} c_{IJ} f_{IJ} \bar{\varphi}^0 \wedge \varphi^I \wedge \bar{\varphi}^J + f_{IJ} \delbar(\varphi^I \wedge \bar{\varphi}^J) =\\
= & \frac{1}{2} c_{IJ} f_{IJ} \bar{\varphi}^0 \wedge \varphi^I \wedge \bar{\varphi}^J + \\
 & + f_{IJ} \left( \sum_{\sigma = 1}^p (-1)^{\sigma - 1} \varphi^{i_1} \wedge \ldots \wedge \left( -\frac{1}{2} \lambda_{i_\sigma} \bar{\varphi}^0 \wedge \varphi^{i_\sigma} \right) \wedge \ldots \wedge \varphi^{i_p} \wedge \bar{\varphi}^J + \right.\\
 & \left. + (-1)^p \varphi^I \wedge \sum_{\tau = 1}^q (-1)^{\tau - 1} \bar{\varphi}^{j_1} \wedge \ldots \wedge \left( -\frac{1}{2} \lambda_{j_\tau} \bar{\varphi}^0 \wedge \bar{\varphi}^{j_\tau} \right) \wedge \ldots \wedge \bar{\varphi}^{j_q} \right) =\\
= & \frac{1}{2} c_{IJ} f_{IJ} \bar{\varphi}^0 \wedge \varphi^I \wedge \bar{\varphi}^J +\\
 & + f_{IJ} \sum_{\sigma = 1}^p \left( -\frac{1}{2} \lambda_{i_\sigma} \right) \bar{\varphi}^0 \wedge \varphi^I \wedge \bar{\varphi}^J +\\
 & + f_{IJ} \sum_{\tau = 1}^q \left( -\frac{1}{2} \lambda_{j_\tau} \right) \bar{\varphi}^0 \wedge \varphi^I \wedge \bar{\varphi}^J =\\
= & 0.
\end{array}\]
As $\delbar \varphi^0 = \delbar \bar{\varphi}^0 = 0$, this implies that
\[\begin{array}{l}
\delbar(f_{IJ} \cdot \varphi^0 \wedge \varphi^I \wedge \bar{\varphi}^J) = -\varphi^0 \wedge \delbar(f_{IJ} \cdot \varphi^I \wedge \bar{\varphi}^J) = 0,\\
\delbar(f_{IJ} \cdot \bar{\varphi}^0 \wedge \varphi^I \wedge \bar{\varphi}^J) = -\bar{\varphi}^0 \wedge \delbar(f_{IJ} \cdot \varphi^I \wedge \bar{\varphi}^J) = 0,\\
\delbar(f_{IJ} \cdot \varphi^0 \wedge \bar{\varphi}^0 \wedge \varphi^I \wedge \bar{\varphi}^J) = \varphi^0 \wedge \bar{\varphi}^0 \wedge \delbar(f_{IJ} \cdot \varphi^I \wedge \bar{\varphi}^J) = 0,
\end{array}\]
hence the proposition.
\end{proof}

\begin{theorem}\label{thm: dolbeault cohomology of N}
Let $N$ be a generalized Nakamura manifold, then
\[H^{p, q}_{\delbar}(N) \simeq B^{p, q}.\]
\end{theorem}
\begin{proof}
By \cite[Corollary 4.2]{Kasuya}, the inclusion of differential bigraded algebras $B^{*, *} \subseteq \Lambda^{*, *} N$ is a quasi-isomorphism (with respect to $\delbar$). By Proposition \ref{prop: delbar trivial} the differential $\delbar$ is trivial on $B^{*, *}$ hence
\[H^{p, q}_{\delbar}(N) = H^q((\Lambda^{p, *} N, \delbar)) \simeq H^q((B^{p, *}, 0)) = B^{p, q}.\]
\end{proof}

\begin{rem}
By our assumption that none of the $\lambda_i$'s vanish we deduce that $\{ I, J \text{ such that } c_{IJ} \neq 0 \} \neq \varnothing$, so we can choose $\tau \in \R \smallsetminus \{ 0 \}$ with
\begin{equation}\label{eq: condition}
|\tau| < \frac{2\pi}{M}, \qquad M = \max_{I, J} \{|c_{IJ}| \text{ such that } c_{IJ} \neq 0 \}.
\end{equation}
With this choice of $\tau$, let $I$, $J$ be multiindices such that $\tau \cdot c_{IJ} \in 2\pi \cdot \Z$, say $\tau \cdot c_{IJ} = 2\pi l$, and assume that $c_{IJ} \neq 0$. Then $l \neq 0$ and
\[|\tau| = \frac{2\pi}{|c_{IJ}|} |l| \geq \frac{2\pi}{M} |l| \geq \frac{2\pi}{M} > |\tau|,\]
which is an absurd. This shows that it is possible to choose $\tau \in \R \smallsetminus \{ 0 \}$ satisfying the following assumption: \begin{equation}\label{eq: technical condition}
\begin{array}{c}
\text{for every $I$, $J$ it holds that}\\
\tau \cdot c_{IJ} \in 2\pi \cdot \Z \Longleftrightarrow c_{IJ} = 0.
\end{array}
\end{equation}
\end{rem}

\begin{prop}\label{prop: conjugation}
Let $N = N_{M, P, \tau}$ be a Nakamura manifold corresponding to a choice of $\tau$ satisfying \eqref{eq: technical condition}. Then the Dolbeault cohomology of $N$ is generated by the classes of some invariant forms and complex conjugation defines an isomorphism $\overline{H^{p, q}_{\delbar}(N)} \simeq H^{q, p}_{\delbar}(N)$.
\end{prop}
\begin{proof}
Observe that in this case we have $f_{IJ} \equiv 1$, hence it follows from \eqref{eq: generators of B^{p, q}} and Theorem \ref{thm: dolbeault cohomology of N} that $H^{p, q}_{\delbar}(N)$ is generated by invariant classes.

Assume now that $\varphi^I \wedge \bar{\varphi}^J$ is a $(p, q)$-form corresponding to $I, J$ such that $c_{IJ} = 0$. Then
\[\overline{\varphi^I \wedge \bar{\varphi}^J} = (-1)^{pq} \varphi^J \wedge \bar{\varphi}^J\]
is a $(q, p)$-form which satisfies $c_{JI} = 0$, hence it is one of the distinguished generators for $H^{q, p}_{\delbar}(N)$. A similar result holds also for the other types of generators in \eqref{eq: generators of B^{p, q}}, hence complex conjugation is well defined on $\bigoplus_{p + q = k} H^{p, q}_{\delbar}(N)$ and sends $H^{p, q}_{\delbar}(N)$ to $H^{q, p}_{\delbar}(N)$.
\end{proof}

\begin{prop}\label{prop: Dolbeault numbers of N}
Let $N = N_{M, P, \tau}$ be a Nakamura manifold corresponding to a choice of $\tau$ satisfying \eqref{eq: technical condition}. Then
\[h^{p, q}_{\delbar}(N) = \# \left\{ (I, J) \text{ such that } c_{IJ} = 0 \text{ with } \begin{array}{lll}
|I| = p, & |J| = q & \text{or}\\
|I| = p - 1, & |J| = q & \text{or}\\
|I| = p, & |J| = q - 1 & \text{or}\\
|I| = p - 1, & |J| = q - 1
\end{array}\right\}.\]
\end{prop}

\subsection{de Rham cohomology}\label{sect: de Rham}
Write $w = u + \sqi v$, $z_i = x_i + \sqi y_i$ (for $i = 1, \ldots, n$) and consider the (real) $1$-forms $\eta^0, \psi^0, \eta^1, \ldots, \eta^n, \psi^1, \ldots, \psi^n$ defined by
\[\varphi^0 = \eta^0 + \sqi \psi^0, \qquad \varphi^i = \eta^i + \sqi \psi^i \qquad i = 1, \ldots, n.\]
Let $e_0, f_0, e_1, \ldots, e_n, f_1, \ldots, f_n$ be the dual vector fields of $\eta^0, \psi^0, \eta^1, \ldots, \eta^n, \psi^1, \ldots, \psi^n$, then the only non-trivial commutators of these vector fields are
\[[e_0, e_i] = \lambda_i e_i, \qquad [e_0, f_i] = \lambda_i f_i \qquad i = 1, \ldots, n.\]
We will still call $e_0, f_0, e_1, \ldots, e_n, f_1, \ldots, f_n$ the corresponding basis for the Lie algebra $\frg$ of the Lie group $G$, then we observe the following facts.

\begin{enumerate}
\item $[\frg, \frg] = \Span\{ e_1, \ldots, e_n, f_1, \ldots, f_n \}$.
\item $[\frg, [\frg, \frg]] = \Span\{ e_1, \ldots, e_n, f_1, \ldots, f_n \} = [\frg, \frg]$ and so $\frg$ can not be nilpotent.
\item $[[\frg, \frg], [\frg, \frg]] = 0$ and so $\frg$ is solvable.
\item Let $v = \mu_0 e_0 + \nu_0 f_0 + \mu_1 e_1 + \ldots + \mu_n e_n + \nu_1 f_1 + \ldots + \nu_n f_n \in \frg$, then the map $ad(v) = [v, \cdot]$ is represented by the matrix:
\[\left( \begin{array}{cc|ccc|ccc}
0 & 0 & 0 & \cdots & 0 & 0 & \cdots & 0\\
0 & 0 & 0 & \cdots & 0 & 0 & \cdots & 0\\
\hline
-\lambda_1 \mu_1 & 0 & \lambda_1 \mu_0 & & & 0 & \cdots & 0\\
\vdots & \vdots & & \ddots & & \vdots & \ddots & \vdots\\
-\lambda_n \mu_n & 0 & & & \lambda_n \mu_0 & 0 & \cdots & 0\\
\hline
-\lambda_1 \nu_1 & 0 & 0 & \cdots & 0 & \lambda_1 \nu_0 & & \\
\vdots & \vdots & \vdots & \ddots & \vdots & & \ddots & \\
-\lambda_n \nu_n & 0 & 0 & \cdots & 0& & & \lambda_n \nu_0
\end{array} \right),\]
whose eigenvalues are all real (observw that it is a lower triangular matrix). As a consequence, $\frg$ is completely solvable and so the de Rham cohomology of a generalized Nakamura manifold $N$ coincides with the Chevalley--Eilenberg cohomology of $\frg$ by \cite[Corollary 4.2]{Hattori}.
\end{enumerate}

\begin{prop}\label{prop: ker d}
Let $N = N_{M, P, \tau}$ be a Nakamura manifold corresponding to a choice of $\tau$ satisfying \eqref{eq: technical condition}. Then
\[\ker \left( d: \Lambda^k T^*_{\C}N \longrightarrow \Lambda^{k + 1} T^*_{\C}N \right) = \bigoplus_{p + q = k} B^{p, q} \oplus \operatorname{im} \left( d: \Lambda^{k - 1} T^*_{\C}N \longrightarrow \Lambda^k T^*_{\C}N \right).\]
\end{prop}
\begin{proof}
By the previous argument we can work with invariant forms, i.e., on $\frg^*_{\C}$. The space $\Lambda^k \frg^*_{\C}$ is generated by
\begin{enumerate}
\item $\varphi^I \wedge \bar{\varphi}^J$ with $|I| + |J| = k$;
\item $\varphi^0 \wedge \varphi^I \wedge \bar{\varphi}^J$ with $|I| + |J| = k - 1$;
\item $\bar{\varphi}^0 \wedge \varphi^I \wedge \bar{\varphi}^J$ with $|I| + |J| = k - 1$;
\item $\varphi^0 \wedge \bar{\varphi}^0 \wedge \varphi^I \wedge \bar{\varphi}^J$ with $|I| + |J| = k - 2$.
\end{enumerate}
Using the structure equation we easily compute that
\begin{enumerate}
\item $d \left( \varphi^I \wedge \bar{\varphi}^J \right) = -\frac{1}{2} c_{IJ} \left( \varphi^0 + \bar{\varphi}^0 \right) \wedge \varphi^I \wedge \bar{\varphi}^J$;
\item $d \left( \varphi^0 \wedge \varphi^I \wedge \bar{\varphi}^J \right) = \frac{1}{2} c_{IJ} \varphi^0 \wedge \bar{\varphi}^0 \wedge \varphi^I \wedge \bar{\varphi}^J$;
\item $d \left( \bar{\varphi}^0 \wedge \varphi^I \wedge \bar{\varphi}^J \right) = -\frac{1}{2} c_{IJ} \varphi^0 \wedge \bar{\varphi}^0 \wedge \varphi^I \wedge \bar{\varphi}^J$;
\item $d \left( \varphi^0 \wedge \bar{\varphi}^0 \wedge \varphi^I \wedge \bar{\varphi}^J \right) = 0$.
\end{enumerate}
As a consequence, a basis for $\ker \left( d: \Lambda^k \frg^*_{\C} \longrightarrow \Lambda^{k + 1} \frg^*_{\C} \right)$ is provided by the forms of type
\begin{enumerate}
\item $\varphi^I \wedge \bar{\varphi}^J$ such that $c_{IJ} = 0$ and $|I| + |J| = k$;
\item $\varphi^0 \wedge \varphi^I \wedge \bar{\varphi}^J$ such that $c_{IJ} = 0$ and $|I| + |J| = k - 1$;
\item $\bar{\varphi}^0 \wedge \varphi^I \wedge \bar{\varphi}^J$ such that $c_{IJ} = 0$ and $|I| + |J| = k - 1$;
\item $\left( \varphi^0 + \bar{\varphi}^0 \right) \wedge \varphi^I \wedge \bar{\varphi}^J$ such that $c_{IJ} \neq 0$ and $|I| + |J| = k - 1$;
\item $\varphi^0 \wedge \bar{\varphi}^0 \wedge \varphi^I \wedge \bar{\varphi}^J$ such that $|I| + |J| = k - 2$;
\end{enumerate}
while a basis for $\operatorname{im} \left( d: \Lambda^{k - 1} \frg^*_{\C} \longrightarrow \Lambda^k \frg^*_{\C} \right)$ is
\begin{enumerate}
\item $\left( \varphi^0 + \bar{\varphi}^0 \right) \wedge \varphi^I \wedge \bar{\varphi}^J$ such that $c_{IJ} \neq 0$ and $|I| + |J| = k - 1$;
\item $\varphi^0 \wedge \bar{\varphi}^0 \wedge \varphi^I \wedge \bar{\varphi}^J$ such that $c_{IJ} \neq 0$ and $|I| + |J| = k - 2$.
\end{enumerate}
Hence the proposition follows.
\end{proof}

\subsection{The \texorpdfstring{$\del\delbar$}{del-delbar}-Lemma}\label{sect: del-delbar-lemma}

In this section we show that the $\del\delbar$-Lemma holds for a Nakamura manifold if the parameter $\tau$ defining the lattice $\Gamma'_\tau$ is properly chosen. This shows item \eqref{item: del-delbar-lemma} of Theorem \ref{thm: main}.

By \cite[Paragraph (5.21)]{DGMS} a necessary condition for the $\del\delbar$-Lemma to hold on a compact complex manifold is that the Fr\"olicher spectral sequence degenerates at $E_1$. We state this fact as a separate result.

\begin{prop}
Let $N = N_{M, P, \tau}$ be a Nakamura manifold corresponding to a choice of $\tau$ satisfying \eqref{eq: technical condition}. Then the Fr\"olicher spectral sequence of $N$ degenerates at $E_1$.
\end{prop}
\begin{proof} 
The first page of the Fr\"olicher spectral sequence is given by
\[\left( E_1^{p, q}, d_1 \right) = \left( H^{p, q}_{\delbar}(N), \del \right).\]
By Theorem \ref{thm: dolbeault cohomology of N} we can identify $H^{p, q}_{\delbar}(N)$ with $B^{p, q}$ and we know from Proposition \ref{prop: ker d} that the generators of $B^{p, q}$ are $d$-closed. In particular, they are also $\del$-closed and so $\del \equiv 0$ on $H^{p, q}_{\delbar}(N)$.
\end{proof}

\begin{theorem}\label{thm: N satisfies the del-delbar-lemma}
Let $N = N_{M, P, \tau}$ be a Nakamura manifold corresponding to a choice of $\tau$ satisfying \eqref{eq: technical condition}. Then $N$ satisfies the $\del\delbar$-Lemma and it has no K\"ahler structure.
\end{theorem}
\begin{proof} First of all, as already remarked, since $N$ is a compact solvmanifold of completely solvable type different from a torus, in view of Hasegawa \cite[Main Theorem]{Hasegawa} $N$ has no K\"ahler structure.

Let us show that $N$ satisfies the $\del\delbar$-Lemma.
This is an application of \cite[Corollary 8 and Definition 4]{ASTT}.

By Proposition \ref{prop: conjugation} we know that complex conjugation on $H^{p, q}_{\delbar}(N) \simeq B^{p, q}$ is a well defined isomorphism onto $H^{q, p}_{\delbar}(N)$.

By Proposition \ref{prop: ker d} we know that
\[H^k_{dR}(N, \C) = \frac{\ker \left( d: \Lambda^k T^*_{\C}N \longrightarrow \Lambda^{k + 1} T^*_{\C}N \right)}{\operatorname{im} \left( d: \Lambda^{k - 1} T^*_{\C}N \longrightarrow \Lambda^k T^*_{\C}N \right)} = \bigoplus_{p + q = k} B^{p, q} \simeq \bigoplus_{p + q = k} H^{p, q}_{\delbar}(N),\]
so that (1), (2) and (3) of Section \ref{preliminaries} hold. This concludes the proof.




\end{proof}

\begin{rem}
In \cite[Theorem 3.4]{Kasuya1} Kasuya proved that under some suitable assumption on the solvable Lie group $G=\C^n\ltimes_\phi\C^m$ it is possible to define a Hermitian metric on a compact quotient of $G$ so that the Hodge symmetry and decomposition hold for Dolbeault and de Rham harmonic forms. In particular, such manifolds satisfy the $\del\delbar$-Lemma.
\end{rem}

\begin{rem}
In \cite[Proposition 4.3]{COUV} the authors constructed a $3$-dimensional compact  complex manifold $M$ such that 
\begin{enumerate}
    \item the Fr\"olicher spectral sequence of $M$ degenerates at $E_1$\\
    \item $h^{p,q}_{\delbar}(M)
    =h^{q,p}_{\delbar}(M)$
\end{enumerate}
but $M$ does not satisfy the $\del\delbar$-Lemma. 
\end{rem}
\section{Fibred products and explicit examples}\label{sect: examples}

Let $N = N_{M, P, \tau}$ be the generalized Nakamura manifold associated to the group $G = G_M = \C \ltimes_{\rho} \C^n$ and its lattice $\Gamma = \Gamma_{P, \tau} = \Gamma'_\tau \ltimes_{\rho} \Gamma''_P$ as in Section \ref{sect: construction}. The short exact sequence
\[1 \longrightarrow \C^n \longrightarrow G \longrightarrow \C \longrightarrow 1\]
and the projection $G \longrightarrow \C$ is equivariant map with respect to the action of left translations by elements in $\Gamma$ on $G$ and in $\Gamma'$ on $\C$.
This means that it is induced a map
\[\begin{array}{rccc}
\pi: & N = \Gamma \backslash G & \longrightarrow & E_\tau = \Gamma'_\tau \backslash \C\\
 & \Gamma \cdot (w, z_1, \ldots, z_n) & \longmapsto & \Gamma'_\tau + w
\end{array}\]
which is a holomorphic fibration on $E$ with fibre the $n$-dimensional torus $F_P \simeq \Gamma''_P \backslash \C^n$.

Consider now two generalized Nakamura manifolds, associated to the groups $G_i = \C \ltimes_{\rho_i} \C^{n_i}$ and lattices $\Gamma_i = \Gamma'_\tau \ltimes_{\rho_i} \Gamma''_i$ for $i = 1, 2$. Observe that with this choice of $\Gamma_1$ and $\Gamma_2$ both $N_1$ and $N_2$ fibre over the same complex torus $E_\tau = \Gamma'_\tau \backslash \C$, so it is possible to consider their fibre product
\begin{equation}\label{eq: fibre product}
\xymatrix{N_1 \times_E N_2 \ar[r] \ar[d] & N_1 \ar[d]^{\pi_1}\\
N_2 \ar[r]^{\pi_2} & E_\tau.}
\end{equation}
The maps $N_1 \times_{E_\tau} N_2 \longrightarrow N_1$ and $N_1 \times_{E_\tau} N_2 \longrightarrow N_1$ are fibrations of fibre $F_2 = \Gamma''_2 \backslash \C^{n_2}$ and $F_1 = \Gamma''_1 \backslash \C^{n_1}$ respectively.

\begin{prop}\label{prop: fibre product}
Let $N = N_1 \times_{E_\tau} N_2$ be as above. Then $N$ is a generalized Nakamura manifold.
\end{prop}
\begin{proof}
Indeed, observe that we have an isomorphism
\[\begin{array}{ccc}
G_1 \times_\C G_2 & \longrightarrow & \C \ltimes_\rho \C^{n_1 + n_2}\\
((w, z^{(1)}_1, \ldots, z^{(1)}_{n_1}), ((w, z^{(2)}_1, \ldots, z^{(2)}_{n_2}))) & \longmapsto & (w, z^{(1)}_1, \ldots, z^{(1)}_{n_1}, z^{(2)}_1, \ldots, z^{(2)}_{n_2})
\end{array}\]
where
\[\begin{array}{rccl}
\rho: & \C & \longmapsto & \GL(n_1 + n_2, \C)\\
 & w & \longmapsto & \diag\left( e^{\frac{1}{2}\lambda^{(1)}_1 (w + \bar{w})}, \ldots, e^{\frac{1}{2}\lambda^{(1)}_n (w + \bar{w})}, e^{\frac{1}{2}\lambda^{(2)}_1 (w + \bar{w})}, \ldots, e^{\frac{1}{2}\lambda^{(2)}_n (w + \bar{w})} \right).
\end{array}\]
Via this isomorphism $\Gamma_1 \ltimes_\C \Gamma_2$ corresponds to $\Gamma' \ltimes_\rho (\Gamma''_1 \times \Gamma''_2)$ and so $N = (\Gamma' \ltimes_\rho (\Gamma''_1 \times \Gamma''_2)) \backslash (G_1 \ltimes_\rho G_2)$ is a generalized Nakamura manifold.
\end{proof}

\begin{rem}
Let $N$ be as in Proposition \ref{prop: fibre product}. Then $N$ fibres over the same genus $1$ curve $E_\tau$ as $N_1$ and $N_2$ and the map $N \longrightarrow E_\tau$ is the composition of the arrows in \eqref{eq: fibre product}.
\end{rem}

\begin{ex}\label{example: generalized Nakamura 3-folds}
Let $A_m \in \SL(2, \Z)$ be the matrix
\[A_m = \left( \begin{array}{cc}
m & m^2 - 1\\
1 & m
\end{array} \right), \qquad m \in \Z_{> 1}.\]
Then $A_m$ has two distinct real eigenvalues
\[m + \sqrt{m^2 - 1} \qquad \text{and} \qquad m - \sqrt{m^2 - 1}\]
and we can use
\[\lambda_1 = \log(m + \sqrt{m^2 - 1}), \qquad \lambda_2 = \log(m - \sqrt{m^2 - 1}), \qquad P = \frac{1}{2 \sqrt{m^2 - 1}} \left( \begin{array}{cc}
1 & \sqrt{m^2 - 1}\\
-1 & \sqrt{m^2 - 1}
\end{array} \right),\]
and $\mu \in \left( 0, -\frac{\pi}{\log(m - \sqrt{m^2 - 1})} \right)$ to produce examples of $3$-dimensional generalized Nakamura manifolds.
\end{ex}

\begin{ex}\label{example: generalized Nakamura 4-folds}
Let $A \in \SL(3, \Z)$ be the symmetric matrix
\[A_m = \left( \begin{array}{ccc}
2 & 1 & -2\\
1 & 1 & -1\\
-2 & -1 & 3
\end{array} \right), \qquad m \in \Z_{> 1}.\]
Then $A$ has three distinct positive real eigenvalues, and so it can be used to produce an example of a $4$-dimensional Nakamura manifold.
\end{ex}

\begin{ex}
Let now $n \in \Z_{\geq 2}$. We can produce an $(n + 1)$-dimensional generalized Nakamura manifold $N$ as follows using Proposition \ref{eq: fibre product}:
\begin{itemize}
\item if $n$ is even then $N$ is the fibre product of $\frac{n}{2}$ generalized Nakamura $3$-folds from Example \ref{example: generalized Nakamura 3-folds};
\item if $n$ is odd then $N$ is the fibre product of $\frac{n - 3}{2}$ generalized Nakamura $3$-folds from Example \ref{example: generalized Nakamura 3-folds} and a generalized Nakamura $4$-fold from Example \ref{example: generalized Nakamura 4-folds}.
\end{itemize}
\end{ex}



\begin{thebibliography}{99}

\bibitem{AG} E. Abbena, A. Grassi, Hermitian Left Invariant Metrics on Complex Lie Groups and Cosymplectic Hermitian Manifolds, {\em Boll. Un. Mat. Ital. A (6)} {\bf 5} (1986), 371--379.

\bibitem{AA} L. Alessandrini, M. Andreatta, Closed transverse $(p,p)$-forms on compact complex manifolds, {\em Compositio Math.} \textbf{61} (1987), 181--200; Erratum \textbf{63} (1987), 143. 

\bibitem{AB1} L.\ Alessandrini, G.\ Bassanelli, {\it Compact $p$-K\"ahler manifolds}.

\bibitem{AB2} L.\ Alessandrini, G.\ Bassanelli, {\it Positive $\del\delbar$-currents and non-K\"ahler geometry}.

\bibitem{ASTT} D. Angella, T. Suwa, N. Tardini, A. Tomassini, Note on Dolbeault cohomology and Hodge structures up to bimeromorphisms, {\em Complex Manifolds} {\bf 7} (2020), 194--214. https://doi.org/10.1515/coma-2020-0103

\bibitem{COUV} M. Ceballos, A. Otal, L. Ugarte, R. Villacampa, Invariant complex structures on $6$-nilmanifolds: classication, Fr\"olicher spec-
tral sequence and special Hermitian metrics, {\em J. Geom. Anal.} {\bf 26} (2016), no. 1, 252--286.

\bibitem{DGMS} P.\ Deligne, P.\ Griffiths, J.\ Morgan and D.\ Sullivan, Real homotopy theory of K\"ahler manifolds, {\em Invent. Math.} {\bf 29} (1975), no. 3, 245--274.

\bibitem{FM} A. Fino, A. Mainenti, A note on $p$-K\"ahler structures on compact quotients of Lie groups, {\tt arXiv:2310.16726 [math.DG]}, to appear in the {\em Ann. Mat. Pura ed Appl.} 

\bibitem{FOU} A. Fino, A. Otal, L. Ugarte,Six-Dimensional Solvmanifolds with Holomorphically Trivial Canonical 
Bundle, 
{\em Internat. Math. Res. Not.}, {\bf 2015}, Issue 24, (2015), pp. 13757–13799, https://doi.org/10.1093/imrn/rnv112


\bibitem{Hasegawa} K.\ Hasegawa. A note on compact solvmanifolds with Kähler structures, {\em Osaka J. Math.} {\bf 43} (2006), no. 1, 131--135

\bibitem{Hattori} A.\ Hattori, Spectral sequence in the de Rham cohomology of fibre bundles, {\em J. Fac. Sci. Univ. Tokyo Sect. I} {\bf 8} (1960), 289--331.

\bibitem{HL} R. Harvey, J.~B. Lawson, An intrinsic charactherization of K\" ahler manifolds,
{\em Inv. Math.} {\bf 74} (1983), 169--198.

\bibitem{HMT} R. K. Hind, C. Medori, A. Tomassini, Families of Almost Complex Structures and Transverse $(p, p)$-Forms, {\em J. Geom. Anal.} {\bf 33}, 334 (2023). https://doi.org/10.1007/s12220-023-01391-x

\bibitem{Kasuya} H.\ Kasuya, Techniques of computations of Dolbeault cohomology of solvmanifolds, {\em Math. Z.} {\bf 273} (2013), no. 1-2, 437--447.

\bibitem{Kasuya1} H.\ Kasuya, Hodge symmetry and decomposition on non-K\"ahler solvmanifolds, {\em J. Geom. Phys.} {\bf 76} (2014), 61--65.
\bibitem{Michelsohn} M.\ L.\ Michelsohn, On the existence of special metrics in complex geometry, {\em Acta Math.} {\bf 143} (1982), 261--295

\bibitem{N} I. Nakamura, Complex parallelisable manifolds and their small deformations, {\em J. Diff. Geom.} {\bf 10} (1975),  85--112.

\bibitem{RWZ} S. Rao, X. Wan, Q. Zhao, On local stabilities of $p$-K\"ahler structures, {\em Compositio Math.} \textbf{155} (2019), 455--483.

\bibitem{S} D. Sullivan, Cycles for the dynamical study of foliated manifolds and complex manifolds, {\em Invent. Math.} \textbf{36} (1976), 225--255.


\bibitem{W} H.-C. Wang,
Complex parallisable manifolds,
{\em Proc. Amer. Math. Soc.} {\bf 5} (1954), 771--776.

\end{thebibliography}
\end{document}